\documentclass[a4paper,draft]{amsproc}
\usepackage{amssymb}
\usepackage[hyphens]{url} 

\theoremstyle{plain}
 \newtheorem{thm}{Theorem}[section]
 
 \newtheorem{lem}{Lemma}[section]
 
\theoremstyle{definition}

\theoremstyle{remark}

 \numberwithin{equation}{section}

\renewcommand{\leq}{\leqslant}
\renewcommand{\setminus}{\smallsetminus}
\newcommand\mc{\mathcal}
\newcommand\mb{\mathbb}

\newcommand\ve{\varepsilon}%{\epsilon}
\newcommand\xt{\mathfrak{T}}

\DeclareMathOperator{\Span}{Span}
\DeclareMathOperator{\Ker}{Ker}
\DeclareMathOperator{\tr}{tr}

\DeclareMathOperator{\id}{Id}
\DeclareMathOperator{\sign}{sgn}

\title[The Jacobi-orthogonality in indefinite scalar product spaces]{THE JACOBI-ORTHOGONALITY IN INDEFINITE SCALAR PRODUCT SPACES}



\author[Luki\'c]{\bfseries Katarina Luki\'c}

\address{
Faculty of Mathematics \\ % \hfill (Received 00 00 202?)\\
University of Belgrade  \\ %\hfill (Revised  00 00 202?)\\
Belgrade\\
Serbia}
\email{katarina.lukic@matf.bg.ac.rs}

\thanks{Supported by the Ministry of Education, Science and Technological Developments of the Republic of Serbia: grant number 451-03-68/2022-14/200104. } 
\thanks{Communicated by ...} 

\begin{document}

{\begin{flushleft}\baselineskip9pt\scriptsize
%PUBLICATIONS DE L'INSTITUT MATH\'EMATIQUE\\
%Nouvelle s\'erie, tome 1??(1??) (202?), od--do \hfill DOI: \\
MANUSCRIPT
\end{flushleft}}
\vspace{18mm} \setcounter{page}{1} \thispagestyle{empty}

\begin{abstract}
We generalize the property of Jacobi-orthogonality to indefinite scalar product spaces. We compare various principles and investigate relations between Osserman, Jacobi-dual, and Jacobi-orthogonal algebraic curvature tensors. We show that every quasi-Clifford tensor is Jacobi-orthogonal. We prove that a Jacobi-diagonalizable Jacobi-orthogonal tensor is Jacobi-dual whenever $\mc{J}_X$ has no null eigenvectors for all nonnull $X$. We show that any algebraic curvature tensor of dimension $3$ is Jacobi-orthogonal if and only if it is of constant sectional curvature. We prove that every $4$-dimensional Jacobi-diagonalizable algebraic curvature tensor is Jacobi-orthogonal if and only if it is Osserman.
\end{abstract}

\maketitle

\section{Introduction}  

Recently, Jacobi-orthogonal algebraic curvature tensors have been introduced as a new potential characterization of Riemannian Osserman tensors, and it has been proved that any Jacobi-orthogonal tensor is Osserman, while all known Osserman tensors are Jacobi-orthogonal \cite{AL3}. The aim of this paper is to generalize the concept of Jacobi-orthogonality to indefinite scalar product spaces and investigate its relations with some important features such as Osserman, quasi-Clifford, and Jacobi-dual tensors.

Let $(\mc{V},g)$ be a scalar product space of dimension $n$, that is, $\mc{V}$ is an $n$-dimensional vector space over $\mathbb{R}$, while $g$ is a nondegenerate symmetric bilinear form on $\mc{V}$. The
sign of the squared norm, $\ve_X=g(X,X)$, distinguishes all vectors $X \in \mc{V}\setminus \{0\}$ into three different types. A vector $X \in \mc{V}$ is spacelike if $\ve_X > 0$; timelike if $\ve_X < 0$; null if $\ve_X=0$ and $X \neq 0$.
Especially, a vector $X\in\mc{V}$ is nonnull if $\ve_X \neq 0$ and it is unit if $\ve_X \in \{-1, 1\}$. We say that $X$ and $Y$ are mutually orthogonal and write $X \perp Y$ if $g(X, Y)=0$. For $X \perp Y$ we have
\begin{equation}\label{eps}
\ve_{\alpha X + \beta Y}=g(\alpha X + \beta Y, \alpha X + \beta Y)=\alpha^2\ve_X+\beta^2\ve_Y.
\end{equation}
One important relation between null, timelike, and spacelike vectors is given in the following lemma (see \cite[Lemma 1]{VKr}).
\begin{lem}\label{Lema59}
    Every null $N$ from a scalar product space $\mc{V}$ can be decomposed
as $N = S + T$, where $S, T \in \mc{V}$, $S \perp T$, and $\ve_S=-\ve_T$.
\end{lem}
We say that a subspace $W$ of an indefinite scalar product space $(\mc{V}, g)$ is totally isotropic if it consists only of null vectors, which implies that any two vectors from $W$ are mutually orthogonal. In this paper we shall use the following well-known statement about an isotropic supplement of $W$ (see \cite[Proposition 1]{AL1}).
\begin{lem}\label{suplement}
    If $\mc{W} \leq \mc{V}$ is a totally isotropic subspace with a basis $(N_1,..., N_k)$, then there exists a totally isotropic subspace $\mc{U} \leq \mc{V}$, disjoint from $\mc{W}$, with a basis $(M_1, ..., M_k)$, such that $g(N_i, M_j) = \delta_{ij}$ holds for $1 \leq i, j \leq k$.
\end{lem}

A quadri-linear map $R: \mc{V}^4 \mapsto \mb{R}$ is said to be an algebraic curvature tensor on $(\mc{V},g)$ if it satisfies the usual $\mb{Z}_2$ symmetries as well as the first Bianchi identity. More concretely, an algebraic curvature tensor $R\in\xt^0_4(\mc{V})$ has the following properties,
\begin{equation}\label{R1}
R(X,Y,Z,W)=-R(Y,X,Z,W),
\end{equation}
\begin{equation}\label{R2}
R(X,Y,Z,W)=-R(X,Y,W,Z),
\end{equation}
\begin{equation}\label{R4}
R(X, Y,Z,W)=R(Z,W,X,Y),
\end{equation}
\begin{equation}\label{R3}
R(X, Y,Z,W) + R(Y,Z,X,W) + R(Z,X, Y,W) = 0,
\end{equation}
for all $X, Y, Z, W \in \mc{V}$.

The basic example of an algebraic curvature tensor is the tensor $R^1$ of constant sectional curvature $1$, defined by
\begin{equation*}
    R^1(X, Y, Z, W) = g(Y, Z)g(X, W) - g(X, Z)g(Y, W).
\end{equation*}
Furthermore, skew-adjoint endomorphisms $J$ on $\mc{V}$ generate new examples by 
\begin{equation*}
   R^J (X, Y, Z, W) = g(JX, Z)g(JY, W)-g(JY, Z)g(JX, W)+2g(JX, Y)g(JZ,W).
\end{equation*}

A quasi-Clifford family of rank $m$ is an anti-commutative family of skew-adjoint endomorphisms $J_i$, for $1\leq i\leq m$, such that $J_i^2=c_i\id$, for $c_i \in \mathbb{R}$. In other words, a quasi-Clifford family satisfies the Hurwitz-like relations, $J_i J_j+J_j J_i=2\delta_{ij}c_i\id$, for $1\leq i,j\leq m$. 
We say that an algebraic curvature tensor $R$ is quasi-Clifford if 
\begin{equation}\label{clifford}
R=\mu_0 R^1+\sum_{i=1}^m \mu_i R^{J_i},
\end{equation}
for some $\mu_0,\dots,\mu_m\in\mb{R}$, where $J_i$, for $1\leq i \leq m$, is some associated quasi-Clifford family. Especially, $R$ is Clifford if it is quasi-Clifford with $c_i = -1$ for all $1 \leq i \leq m$. Let us remark that Clifford tensors were observed in \cite{G1, G2} and quasi-Clifford tensors were considered in \cite{AL1}.

 If $E_1, E_2, ..., E_n \in \mc{V}$ are mutually orthogonal units, we say that $(E_1, ..., E_n)$ is an orthonormal basis of $\mc{V}$. The signature of a scalar product space $(\mc{V}, g)$ is an ordered pair $(p, q)$, where $p$ is the number of negative $\ve_{E_i}$, while $q$ is the number of positive $\ve_{E_i}$. We say that $R$ is Riemannian if $p=0$; Lorentzian if $p=1$; Kleinian if $p=q$.

Raising the index we obtain the algebraic curvature operator $\mc{R}=R^\sharp\in\xt^1_3(\mc{V})$. The polarized Jacobi operator is the linear map $\mc{J} : \mc{V}^3 \mapsto \mc{V}$ defined by
\begin{equation*}
 \mc{J}(X, Y)Z=\frac{1}{2}(\mc{R}(Z, X)Y+\mc{R}(Z, Y)X)   
\end{equation*}
for all $X, Y, Z \in \mc{V}$.  For each $X \in \mc{V}$ the Jacobi operator $\mc{J}_X$ is a linear self-adjoint operator $\mc{J}_X\colon\mc{V}\to\mc{V}$ defined by $\mc{J}_X Y=\mc{J}(X, X)Y=\mc{R}(Y,X)X$ for all $Y \in \mc{V}$. 
Using the three-linearity of $\mc{R}$, for every $Z \in \mc{V}$ we get
\begin{equation}\label{jakobi_indeks}
\mc{J}_{tX}Z=\mc{R}(Z, tX)(tX)=t^2\mc{R}(Z, X)X=t^2\mc{J}_XZ
\end{equation}
and
\begin{equation}\label{jakobi}
\mc{J}_{X+Y}Z=\mc{R}(Z, X+Y)(X+Y)=\mc{J}_XZ+2\mc{J}(X,Y)Z+\mc{J}_YZ.
\end{equation}

Using \eqref{R4} we get that any two Jacobi operators satisfy the compatibility condition, which means that $g(\mc{J}_XY, Y)=g(\mc{J}_Y X,X)$ holds for all $X, Y \in \mc{V}$. Since $\mc{J}_XX=0$ and $g(\mc{J}_XY,X)=0$, we conclude that for any nonnull $X\in\mc{V}$ the Jacobi operator $\mc{J}_X$ is completely determined by its restriction $\widetilde{\mc{J}}_X\colon X^\perp \to X^\perp$ called the reduced Jacobi operator.

Let $R$ be an algebraic curvature tensor and $\widetilde{w}_X(\lambda)=\det(\lambda \id -\widetilde{\mc{J}}_X)$. We say that $R$ is timelike Osserman if $\widetilde{w}_X$ is independent of unit timelike $X \in \mc{V}$. We say that $R$ is spacelike Osserman if $\widetilde{w}_X$ is independent of unit spacelike $X \in \mc{V}$. Naturally, $R$ is called Osserman if it is both timelike and spacelike Osserman. It is known that timelike Osserman and spacelike Osserman conditions are equivalent (see \cite{Spanci}). It is easy to see that every quasi-Clifford tensor is Osserman (see \cite{AL1}).

We say that $R$ is $k$-stein if there exist constants $c_1,..., c_k \in \mathbb{R}$ such that 
\begin{equation}\label{trag}
   \tr((\mc{J}_X)^j) = (\ve_X)^jc_j 
\end{equation} 
holds for each $1 \leq j \leq k$ and all $X \in \mc{V}$. It is known that an algebraic curvature tensor of dimension $n$ is Osserman if and only if it is $n$-stein (see \cite[Lemma 1.7.3]{G0}).

 We say that $R$ is Jacobi-diagonalizable if $\mc{J}_X$ is diagonalizable for any nonnull $X$. In this case we have
 \begin{equation}\label{velikasuma}
     \mc{V}=\Span\{X\} \oplus \bigoplus_{l=1}^{k}\Ker(\widetilde{\mc{J}}_X-\ve_X \lambda_l \id),
 \end{equation}
 where $\ve_X \lambda_1, ..., \ve_X \lambda_k$ are all eigenvalues of $\widetilde{\mc{J}}_X$ and $\oplus$ denotes the direct orthogonal sum.

The duality principle in the Riemannian setting ($g$ is positive definite) appeared in \cite{Ra2}. Its generalization to a pseudo-Riemannian setting (see \cite{AR1, AR2}) is given by the implication
\begin{equation}\label{Jdual}
\mc{J}_XY=\ve_X \lambda Y \implies \mc{J}_YX=\ve_Y \lambda X.
\end{equation}

If \eqref{Jdual} holds for all mutually orthogonal unit $X, Y \in \mc{V}$ we say that $R$ is weak Jacobi-dual, and if \eqref{Jdual} holds for all $X, Y \in \mc{V}$ with the restriction $\ve_X \neq 0$, we say that $R$ is Jacobi-dual. 
If $R$ is Jacobi-diagonalizable, it is sufficient to prove that it is weak Jacobi-dual which we see in the following lemma (see \cite{VKr, AR1}). 
\begin{lem}\label{ARteo}
Every Jacobi-diagonalizable algebraic curvature tensor is Jacobi-dual if and only if it is weak Jacobi-dual.
\end{lem}
The condition that $R$ is Jacobi-diagonalizable is strong enough to provide the equivalence between Osserman and Jacobi-dual property in a pseudo-Riemannian setting.
\begin{thm}[Nikolayevsky \& Raki\'c \cite{NR2}]\label{NRteo}
Every Jacobi-diagonalizable algebraic curvature tensor is Osserman if and only if it is Jacobi-dual.
\end{thm}

\section{The Jacobi-orthogonality}

In \cite{AL3} we introduced a new concept of Jacobi-orthogonality, and here we generalize it to a pseudo-Riemannian setting. We say an algebraic curvature tensor is Jacobi-orthogonal if the implication 
\begin{equation}\label{ort_impl}
  X \perp Y \implies \mc{J}_XY \perp \mc{J}_YX  
\end{equation} holds for all unit $X, Y \in \mc{V}$. However, it is easy to extend this for all $X, Y \in \mc{V}$, which we see in the following lemma.

\begin{lem}\label{JOzasve}
    If an algebraic curvature tensor is Jacobi-orthogonal, then \eqref{ort_impl} holds for all $X, Y \in \mc{V}$.
\end{lem}

\begin{proof}
    Suppose $R$ is Jacobi-orthogonal and $X \perp Y$. The assertion is obvious for $X=0$ or $Y=0$. If $X$ and $Y$ are both nonnull, \eqref{ort_impl} holds after we rescale them using \eqref{jakobi_indeks}.
    
    We consider the case $\ve_X \neq 0$ and $\ve_Y=0$. Since $X^{\perp}$ is nondegenerate and contains null $Y$, according to Lemma \ref{Lema59}, there exist $S, T \in X^{\perp}$ such that $Y=S+T$, $S \perp T$, $\ve_S=-\ve_T > 0$. Since $X, S, T$ are nonnull, $X \perp S$, and $X \perp T$, using \eqref{ort_impl} we get $g(\mc{J}_XS, \mc{J}_SX)=0$ and $g(\mc{J}_XT, \mc{J}_TX)=0$. Hence, using \eqref{jakobi} and denoting $K=\mc{J}_SX$, $L=2\mc{J}(S, T)X$, $M=\mc{J}_TX$, $P=\mc{J}_XS$, and $Q=\mc{J}_XT$, we calculate
    \begin{equation}\label{duga}
    \begin{aligned}
        g(\mc{J}_X(S+\lambda T)&, \mc{J}_{S+\lambda T}X)=g(P+\lambda Q, K+\lambda L+\lambda^2M)\\
        &=(g(P, M)+g(Q, L))\lambda^2+(g(Q, K)+g(P, L))\lambda.
    \end{aligned}
    \end{equation}
    
     For every $\lambda \neq \pm 1$, using \eqref{eps} we get $\ve_{S+\lambda T}=\ve_S(1-\lambda^2) \neq 0$, so $X \perp S+\lambda T$ implies $g(\mc{J}_X(S+\lambda T), \mc{J}_{S+\lambda T}X)=0$, where \eqref{duga} gives $g(P, M)+g(Q, L)=0$ and $g(Q, K)+g(P, L)=0$. Hence, \eqref{duga} for $\lambda=1$ implies $g(\mc{J}_X(S+T), \mc{J}_{S+T}X)=0$ which proves \eqref{ort_impl} for one nonnull and one null vector.
    
    It remains to prove \eqref{ort_impl} for two null vectors $X=N_1$ and $Y=N_2$.
    If they are linearly dependent, we have $N_1=\xi N_2$ for some $\xi \in \mb{R}$, so $\mc{J}_{N_1}N_2=0$ and therefore \eqref{ort_impl} holds. If $N_1$ and $N_2$ are linearly independent mutually orthogonal vectors, then they form a basis $(N_1, N_2)$ of the totally isotropic subspace $\Span\{N_1, N_2\} \leq \mc{V}$. According to Lemma \ref{suplement} there exists a basis $(M_1, M_2)$ of a totally isotropic subspace of $\mc{V}$ that is disjoint from $\Span\{N_1, N_2\}$ and $g(N_i, M_j)=\delta_{ij}$, for $1 \leq i, j \leq 2$. We can decompose $N_2=S+T$, where $S=(N_2+M_2)/2$, $T=(N_2-M_2)/2$, and $S, T \in N_1^{\perp}$. Since $\ve_S=-\ve_T=1/2$ and $S \perp T$, repeating the same procedure as in the previous part of the proof, we get \eqref{duga} and using already proved implication \eqref{ort_impl} for nonnull $S+\lambda T$ and null vector $N_1$ we have \eqref{ort_impl} for null vectors $X=N_1$ and $Y=N_2$, which completes the proof.   
\end{proof}

Sometimes, it is useful to add the tensor of constant sectional curvature to the observed algebraic curvature tensor $R$. 
\begin{lem}\label{pomeranjeR}
    If an algebraic curvature tensor $R$ is Jacobi-orthogonal, then for each $\mu \in \mathbb{R}$, the tensor $R+\mu R^1$ is Jacobi-orthogonal. 
\end{lem}

\begin{proof}
    Let $\mc{J}'$ be the Jacobi operator associated with the algebraic curvature tensor $R'=R+\mu R^1$, while $X$ and $Y$ are mutually orthogonal unit vectors. Using $\mc{J}_XY \perp X$, $\mc{J}_YX \perp Y$, and the Jacobi-orthogonality of $R$, we get
    \begin{equation*}
        g(\mc{J}'_XY, \mc{J}'_YX)=g(\mc{J}_XY+\mu \ve_X Y, \mc{J}_YX+\mu \ve_Y X)=g(\mc{J}_XY, \mc{J}_YX)=0,
    \end{equation*}
    which means that $R'=R+\mu R^1$ is Jacobi-orthogonal.
\end{proof}
In the Riemannian setting we know that every Clifford algebraic curvature tensor is Jacobi-orthogonal (see \cite{AL3}). We use Lemma \ref{pomeranjeR} to give a generalization to a pseudo-Riemannian setting.

\begin{thm}\label{CliffJO}
Every quasi-Clifford algebraic curvature tensor is Jacobi-or\-thog\-o\-nal.
\end{thm}
\begin{proof}
Let $J_1, J_2, ..., J_m$ be a quasi-Clifford family associated to a quasi-Clifford algebraic curvature tensor of the form \eqref{clifford}. Consider $R=\sum_{i=1}^m \mu_i R^{J_i}$ and units $X \perp Y$. Since the endomorphism $J_i$ is skew-adjoint, we have $g(J_iX, X)=0$, which yields
\begin{equation*}
\begin{aligned} 
\mc{J}_XY&=\sum_{i=1}^m \mu_i \mc{R}^{J_i}(Y, X)X\\
&=\sum_{i=1}^m \mu_i (g(J_iY, X)J_iX-g(J_iX, X)J_iY+2g(J_iY, X)J_iX)\\
&=3 \sum_{i=1}^m \mu_i g(J_iY, X)J_iX,
\end{aligned}
\end{equation*}
and similarly $\mc{J}_YX=3 \sum_{j=1}^m \mu_j g(J_jX, Y)J_jY$. For units $X\perp Y$, using that $J_i$ is skew-adjoint for $i \in \{1,2,...,m\}$ and the Hurwitz-like relations, we get

\begin{equation*}
\begin{aligned}
g(\mc{J}_XY, \mc{J}_YX)
&=g\Big(3 \sum_{i=1}^m \mu_i g(J_iY, X)J_iX, \, 3 \sum_{j=1}^m \mu_j g(J_jX, Y)J_jY\Big)\\
&=9\sum_{i,j}\mu_i\mu_j g(J_iY, X)g(J_jX, Y)g(J_iX,J_jY)\\
&=9\sum_{i,j}\mu_i\mu_j g(X, J_iY)g(X, J_jY)g(X, J_iJ_jY)\\
&=\frac{9}{2}\sum_{i,j}\mu_i\mu_j g(X, J_iY)g(X, J_jY)g(X, (J_iJ_j+J_jJ_i)Y)\\
&=\frac{9}{2}\sum_{i,j} 2\delta_{ij}c_i\mu_i\mu_j g(X, J_iY)g(X, J_jY)g(X, Y)=0,
\end{aligned}
\end{equation*}
which proves that $R$ is Jacobi-orthogonal. According to Lemma \ref{pomeranjeR} it follows that the quasi-Clifford $R+\mu_0R^1$ is Jacobi-orthogonal.
\end{proof}

In order to examine the Jacobi-duality of a Jacobi-diagonalizable Jacobi-ortho\-gonal algebraic curvature tensor, we give the following two lemmas which give us information about $\mc{J}_YX$, where $Y$ is an eigenvector of $\mc{J}_X$ for a nonnull vector $X \in \mc{V}$.

\begin{lem}\label{BitnaLema1}
    Let $R$ be a Jacobi-diagonalizable Jacobi-orthogonal algebraic curvature tensor. If $X \in \mc{V}$ is a nonnull vector and $Y \in \mc{V}_i(X)=\Ker(\widetilde{\mc{J}}_X-\ve_X \lambda_i \id)$, then $\mc{J}_YX \in \Span\{X\} \oplus\mc{V}_i(X)$.
\end{lem}

\begin{proof}
    If $\widetilde{\mc{J}}_X$ has only one eigenvalue $\ve_X \lambda_i$, then $\Span\{X\} \oplus \mc{V}_i(X)=\mc{V}$, so the statement is obvious.
    Let $Z \in \mc{V}_j(X)=\Ker(\widetilde{\mc{J}}_X-\ve_X\lambda_j \id)$ for $\lambda_j \neq \lambda_i$ and $L=Y+tZ$, where $t \in \mb{R}$. Since $Y \in \mc{V}_i(X) \leq X^\perp$ and $Z \in \mc{V}_j(X) \leq X^\perp$ we have $L \perp X$, so using the Jacobi-orthogonality of $R$, Lemma \ref{JOzasve}, and \eqref{R2}, we get 
    \begin{equation*}
        \begin{aligned}
            0=g(\mc{J}_LX, \mc{J}_XL)=&g( \mc{R}(X, Y+tZ)(Y+tZ), \mc{J}_XY+t\mc{J}_XZ)\\
         =&R(X, Y+tZ, Y+tZ, \ve_X \lambda_i Y+t\ve_X \lambda_j Z)\\
         =&\ve_X(t\lambda_j-t\lambda_i)R(X, Y+tZ, Y, Z)\\
       =&\ve_X(\lambda_i-\lambda_j)R(X, Z, Z, Y)t^2+\ve_X(\lambda_j-\lambda_i)R(X, Y, Y, Z)t.  
        \end{aligned}
    \end{equation*}
 Since this holds for all $t \in \mathbb{R}$, we conclude that the coefficient of $t$ is zero and because of $\ve_X (\lambda_j -\lambda_i) \neq 0$ we obtain $R(X, Y, Y, Z)=0$, and therefore $\mc{J}_YX \perp Z$, which holds for every $Z \in \mc{V}_j(X)$, whenever $\lambda_j \neq \lambda_i$. Since $R$ is Jacobi diagonalizable, we have \eqref{velikasuma}, where $\ve_X \lambda_1, ..., \ve_X \lambda_k$ are all (different) eigenvalues of $\widetilde{\mc{J}}_X$, so we conclude that $\mc{J}_YX \in \Span\{X\} \oplus\mc{V}_i(X)$.
\end{proof}

\begin{lem}\label{BitnaLema2}
    Let $R$ be a Jacobi-diagonalizable Jacobi-orthogonal algebraic curvature tensor. If $X \in \mc{V}$ is a nonnull vector and $Y \in \mc{V}(X)=\Ker(\widetilde{\mc{J}}_X-\ve_X \lambda \id)$, then $\mc{J}_YX=\ve_Y \lambda X + Z$, where $\ve_Z=0$.
\end{lem}

\begin{proof}
  Let $\mc{J}_YX=\alpha X+Z$, where $Z \in X^{\perp}$ and $\alpha \in \mathbb{R}$. The compatibility of Jacobi operators gives $g(\mc{J}_YX,X)=g(\mc{J}_XY,Y)$, so $g(\alpha X + Z, X)=g(\ve_X \lambda Y, Y)$. Hence, $\alpha \ve_X=\lambda \ve_X \ve_Y$ and since $\ve_X \neq 0$, we get $\alpha=\ve_Y \lambda$ and $\mc{J}_YX= \ve_Y \lambda X+Z$. From $Y \in \mc{V}(X) \leq X^{\perp}$, we get $g(\ve_X Y-t\ve_Y X, X+tY)=0$, so using that $R$ is Jacobi-orthogonal, Lemma \ref{JOzasve}, \eqref{jakobi_indeks}, \eqref{jakobi}, and the equations $2\mc{J}(X,Y)Y=-\mc{J}_YX$, $2\mc{J}(X,Y)X=-\mc{J}_XY$, we obtain 
  \begin{equation*}
      \begin{aligned}
         0=g(&\mc{J}_{X+tY}(\ve_X Y-t\ve_Y X), \mc{J}_{\ve_X Y-t\ve_Y X} (X+tY))\\
          =g(&\ve_X\mc{J}_XY-t\ve_X\mc{J}_YX+t^2\ve_Y\mc{J}_XY-t^3\ve_Y\mc{J}_YX,\\
            &\ve_X ^2\mc{J}_YX+t\ve_X\ve_Y\mc{J}_XY+t^2\ve_X\ve_Y\mc{J}_YX+t^3\ve_Y ^2\mc{J}_XY).
      \end{aligned}
  \end{equation*}
  Since every $t\in \mathbb{R}$ is a root of the polynomial equation 
  \begin{equation*}
     g(\mc{J}_{X+tY}(\ve_X Y-t\ve_Y X), \mc{J}_{\ve_X Y-t\ve_Y X} (X+tY))=0, 
  \end{equation*}
    we conclude that all coefficients are zero, and therefore the coefficient of $t$ is $\ve_X ^2 \ve_Yg(\mc{J}_XY,\mc{J}_XY)-\ve_X ^3g(\mc{J}_YX,\mc{J}_YX)=0$, which implies $\ve_Y \ve_{\mc{J}_XY}=\ve_X \ve_{\mc{J}_YX}$ because $\ve_X \neq 0$, and therefore $\ve_Y \ve_{\ve_X \lambda Y}=\ve_X \ve_{\ve_Y \lambda X + Z}$. Since $Z \in X^\perp$, using \eqref{eps}, we get $\ve_Y \ve_X ^2 \lambda^2 \ve_Y=\ve_X(\ve_Y ^2 \lambda^2 \ve_X + \ve_Z)$, which gives $\ve_Z=0$.
\end{proof}

As a consequence of the last two lemmas we easily get the following theorem.

\begin{thm} \label{Posledica1} 
Every Jacobi-diagonalizable Jacobi-orthogonal algebraic curvature tensor is Jacobi-dual, when $\mc{J}_X$ has no null eigenvectors for all nonnull $X$.
\end{thm}

\begin{proof}
    Let $X$ and $Y$ be two mutually orthogonal vectors such that $\ve_X \neq 0$ and $\mc{J}_XY=\ve_X \lambda Y$. Using Lemma \ref{BitnaLema2} we get $\mc{J}_YX=\ve_Y \lambda X+Z$, where $\ve_Z=0$, while Lemma \ref{BitnaLema1} gives $Z \in \Ker(\widetilde{\mc{J}}_X-\ve_X \lambda \id)$. If $Z$ is null, then it is not an eigenvector of $\mc{J}_X$, which implies $Z=0$, so $\mc{J}_YX=\ve_Y \lambda X$, which proves that $R$ is Jacobi-dual.
\end{proof}

\section{Low dimensional cases}

In this section we consider the cases of small dimension $n \in \{3, 4\}$. In dimension $3$ we obtain the following expected result.

\begin{thm}
    Every algebraic curvature tensor of dimension $3$ is Jacobi-orthogonal if and only if it is of constant sectional curvature.
\end{thm}

\begin{proof}
    Suppose $R$ is a $3$-dimensional algebraic curvature tensor of constant sectional curvature $\mu$. Since the zero tensor is Jacobi-orthogonal, Lemma \ref{pomeranjeR} implies that $R=0+\mu R^1$ is Jacobi-orthogonal. 

    Conversely, suppose $R$ is a Jacobi-orthogonal algebraic curvature tensor of dimension $3$. Let $(E_1, E_2, E_3)$ be an arbitrary orthonormal basis of $\mc{V}$, $\ve_i=\ve_{E_i}$, for $1 \leq i \leq 3$, and $R_{ijkl}=R(E_i, E_j, E_k, E_l)$, for $i, j, k, l \in \{1,2,3\}$. Using the formula $\mc{R}(E_i, E_j)E_k=\sum_l \ve_l R_{ijkl} E_l$ and \eqref{R2}, we obtain $\mc{J}_{E_1}E_2=\ve_2 R_{2112}E_2+\ve_3 R_{2113} E_3$ and $\mc{J}_{E_2}E_1=\ve_1 R_{1221}E_1+\ve_3R_{1223}E_3$. Hence, since $E_1 \perp E_2$ and $R$ is Jacobi-orthogonal, we get $R_{2113}R_{1223}=0$.
    Using rescaling we obtain 
    \begin{equation}\label{ABCprva}
        R(B, A, A, C)R(A, B, B, C)=0,
    \end{equation}
    for an arbitrary orthogonal basis $(A, B, C)$ which consists of nonnull vectors.
    
    Consider the basis $X=E_1$, $Y=E_2+tE_3$, $Z=t\ve_3 E_2-\ve_2 E_3$, where $t>1$. Using \eqref{eps}, we get $\ve_X=\ve_1\neq 0$, $\ve_Y=\ve_2+t^2\ve_3\neq 0$, $\ve_Z=t^2\ve_3^2\ve_2+\ve_2^2\ve_3 \neq 0$, $g(X, Y)=0$, $g(X, Z)=0$ and $g(Y, Z)=t\ve_3\ve_2-t\ve_2\ve_3=0$, so $(X, Y, Z)$ is an orthogonal basis which consists of nonnull vectors, so applying \eqref{ABCprva} we get 
    \begin{equation*}
    \begin{aligned}
        0&=R(E_2+tE_3, E_1, E_1, t\ve_3 E_2-\ve_2E_3)R(E_1, E_2+tE_3, E_2+tE_3, t \ve_3 E_2-\ve_2 E_3)\\
        &=(-\ve_2R_{2113}+(\ve_3R_{2112}-\ve_2R_{3113})t+\ve_3R_{3112}t^2)(R_{1223}+tR_{1323})(-\ve_2-\ve_3t^2).
        \end{aligned}
    \end{equation*}
    Since this holds for every $t>1$, we conclude that the coefficient of $t$ in the polynomial is $0$. Thus, using \eqref{R2} and $\ve_2\neq 0$, we get
    \begin{equation*}
        \ve_2R_{2113}R_{1332}+(\ve_3R_{2112}-\ve_2R_{3113})R_{1223}=0,
    \end{equation*}
    so \eqref{ABCprva} for $(A,B, C)=(E_3, E_1, E_2)$ implies $(\ve_3R_{2112}-\ve_2R_{3113})R_{1223}=0$.
    Rescaling the vectors we obtain
    \begin{equation}\label{2uokvirenaABC}
        (\ve_CR(B, A, A, B)-\ve_BR(C,A,A,C))R(A,B,B,C)=0,
    \end{equation}
    for an arbitrary orthogonal basis $(A, B, C)$ which consists of nonnull vectors.
    
    Let $(E_1, E_2, E_3)$ be an arbitrary orthonormal basis of $\mc{V}$ and $(p,q,r)$ a permutation of the set $\{1,2,3\}$.
    Let $s_1=R_{2113}$, $s_2=R_{1223}$, $s_3=R_{1332}$, $k_1=\ve_2 \ve_3 R_{3223}$, $k_2=\ve_1 \ve_3 R_{3113}$, and $k_3=\ve_1 \ve_2 R_{2112}$.
    The equation \eqref{ABCprva} for $(A,B,C)=(E_p, E_q, E_r)$ gives $s_ps_q=0$, and since this holds for an arbitrary permutation $(p,q,r)$ of the set $\{1,2,3\}$, we get that at least two of $s_1$, $s_2$, $s_3$ are zero. Let $s_p=s_q=0$ and suppose $s_r \neq 0$.
    Hence, the equation \eqref{2uokvirenaABC} for $(A,B,C)=(E_q, E_r, E_p)$, multiplied by $\ve_p\ve_q\ve_r\neq 0$, gives $(k_p-k_r)s_r=0$, which implies $k_p=k_r$.
    
    Consider $A=E_1+tE_3$, $B=E_2$, $C=\ve_3tE_1-\ve_1E_3$, for $t>1$. Using \eqref{eps} we get $\ve_A=\ve_1+t^2\ve_3 \neq 0$, $\ve_B=\ve_2 \neq 0$, $\ve_C=\ve_3^2t^2\ve_1+\ve_1^2\ve_3=t^2\ve_1+\ve_3 \neq 0$, $g(A, B)=0$, $g(A, C)=\ve_3t\ve_1-t\ve_1\ve_3=0$, and $g(B, C)=0$, so $(E_1+tE_3, E_2, \ve_3tE_1-\ve_1E_3)$ is an orthogonal basis which consists of nonnull vectors and applying \eqref{2uokvirenaABC}, \eqref{eps}, \eqref{R1}, \eqref{R2}, \eqref{R4} we compute
    \begin{equation*}
    \begin{aligned}
       ((t^2\ve_1+\ve_3)&(R_{2112}+2R_{1223}t+R_{3223}t^2)-\ve_2R_{3113}(\ve_1+\ve_3t^2)^2)\cdot\\
        &(-\ve_1R_{1223}+(\ve_3R_{2112}-\ve_1R_{3223})t+\ve_3R_{1223}t^2)=0. 
    \end{aligned}
    \end{equation*}
    This holds for every $t>1$, so the coefficient of $t$ is zero, and using $\ve_1^2 \ve_2^2 \ve_3^2=1$ we obtain
    \begin{equation*}
        -2\ve_1\ve_3R_{1223}^2+(\ve_1\ve_2R_{2112}-\ve_1\ve_3R_{3113})(\ve_1\ve_2R_{2112}-\ve_2\ve_3R_{3223})=0.
    \end{equation*}
    Hence, $-2\ve_1\ve_3s_2^2+(k_3-k_2)(k_3-k_1)=0$.
    Thus, using the basis $(E_q, E_r, E_p)$ instead of $(E_1, E_2, E_3)$, we get 
    \begin{equation}\label{proizvodk}
    -2\ve_q\ve_ps^2_r+(k_p-k_r)(k_p-k_q)=0,
    \end{equation}
     which with $k_p=k_r$ and $\ve_q\ve_p \neq 0$ gives $s_r=0$, which contradicts $s_r\neq 0$. Thus, $s_p=s_q=s_r=0$, which implies 
    \begin{equation*}
        R_{2113}=R_{1223}=R_{1332}=0.
    \end{equation*}
    Hence, \eqref{proizvodk} gives $(k_p-k_r)(k_p-k_q)=0$ for any permutation $(p,q,r)$ of the set $\{1,2,3\}$, so at least two of differences $k_3-k_2$, $k_3-k_1$, and $k_2-k_1$ are zero, which implies $k_1=k_2=k_3=\mu$, and therefore 
    \begin{equation*}
        R_{2112}=\ve_1\ve_2\mu, \, R_{3113=}\ve_1 \ve_3\mu, \, R_{3223}=\ve_2\ve_3\mu.
    \end{equation*}
    Since an algebraic curvature tensor is uniquely determined by its $6$ components of tensor: $R_{2113}$, $R_{1223}$, $R_{1332}$, $R_{2112}$, $R_{3113}$, $R_{3223}$ (see \cite[p. 142-144]{Wag}), the previous equations imply that $R$ is of constant sectional curvature $\mu$.
\end{proof}

Since every $3$-dimensional $R$ is $1$-stein if and only if it is of constant sectional curvature (see \cite[Proposition 1.120]{besse}), the previous theorem implies that every $3$-dimensional $R$ is Jacobi-orthogonal if and only if it is Osserman. In the following theorem we prove a similar result in dimension $4$ using an additional hypothesis that $R$ is Jacobi-diagonalizable.

\begin{thm}
    Every Jacobi-diagonalizable algebraic curvature tensor of dimension $4$ is Osserman if and only if it is Jacobi-orthogonal.
\end{thm}

\begin{proof}
    Suppose $R$ is a Jacobi-diagonalizable Osserman algebraic curvature tensor of dimension $4$. It is well-known that a Lorentzian Osserman algebraic curvature tensor has constant sectional curvature (see \cite{Neda, Spanci2}), so it is of the form $R=\mu R^1$. Hence, using that $0$ is Jacobi-orthogonal and applying Lemma \ref{pomeranjeR}, we conclude that Lorentzian $R$ is Jacobi-orthogonal.
    It remains to deal with a Riemannian or Kleinian $R$. Let $X$ and $Y$ be mutually orthogonal unit vectors in $\mc{V}$. Denote $X=E_1$. Since $R$ is Jacobi-diagonalizable, there exists an orthonormal eigenbasis $(E_1, E_2, E_3, E_4)$ related to $\mc{J}_{E_1}$ such that $\mc{J}_{E_1}E_i=\ve_1 \lambda_i E_i$, for $2 \leq i \leq 4$, where $\ve_j=\ve_{E_j}$, for $1 \leq j \leq 4$. Since $R$ is not Lorentzian, we have $\ve_1 \ve_2\ve_3\ve_4=1$, as well as $\ve_i^2=1$, for $1 \leq i \leq 4$. Denoting $R_{ijkl}=R(E_i, E_j, E_k, E_l)$, we get $R_{i11j}=g(\mc{J}_{E_1}E_i, E_j)=g(\ve_1 \lambda_i E_i, E_j)=\ve_1 \lambda_i \delta_{ij}\ve_i$. Hence,
    \begin{equation}\label{R2112}
       R_{2112}=\ve_1\ve_2 \lambda_2, \quad R_{3113}=\ve_1\ve_3 \lambda_3, \quad R_{4114}=\ve_1\ve_4 \lambda_4,
    \end{equation}
    and
    \begin{equation}\label{R2113}
      R_{2113}=R_{2114}=R_{3114}=0.
    \end{equation}
    According to Theorem \ref{NRteo}, a Jacobi-diagonalizable Osserman $R$ is Jacobi-dual. Thus, $\mc{J}_{E_1}E_i=\ve_1 \lambda_i E_i$, for $2 \leq i \leq 4$, implies $\mc{J}_{E_i}E_1=\ve_i \lambda_i E_1$, so $\mc{J}_{E_i}E_1 \perp E_j$ for $2 \leq j \leq 4$, which means $0=g(\mc{J}_{E_i}E_1, E_j)=R_{1iij}$ and therefore 
    \begin{equation}\label{R1223}
     R_{1223}=R_{1224}=R_{1332}=R_{1334}=R_{1442}=R_{1443}=0.
    \end{equation}
     Since $R$ is $1$-stein, the equation \eqref{trag} holds for $j=1$ and we get $\sum_i \ve_i \ve_x R_{ixxi}=c_1$, for $x \in \{1, 2, 3, 4\}$ (see \cite{VKr}). Thus, using \eqref{R4} we obtain 
    \begin{equation*}
        \begin{aligned}
            &\ve_1 \ve_2 R_{2112}+\ve_1 \ve_3 R_{3113}+\ve_1 \ve_4 R_{4114}=c_1,\\
            &\ve_1 \ve_2 R_{2112}+\ve_2 \ve_3 R_{3223}+\ve_2 \ve_4 R_{4224}=c_1,\\
            &\ve_1 \ve_3 R_{3113}+\ve_2 \ve_3 R_{3223}+\ve_3 \ve_4 R_{4334}=c_1,\\
            &\ve_1 \ve_4 R_{4114}+\ve_2 \ve_4 R_{4224}+\ve_3 \ve_4 R_{4334}=c_1.
        \end{aligned}
    \end{equation*}
    Therefore, subtracting the sum of the two of these equations from the sum of the remaining two equations, we get $\ve_2 \ve_3 R_{3223}=\ve_1 \ve_4 R_{4114}$, $\ve_2 \ve_4 R_{4224}=\ve_1 \ve_3 R_{3113}$, and $\ve_3 \ve_4 R_{4334}=\ve_1 \ve_2 R_{2112}$. Using \eqref{R2112}, we obtain
    \begin{equation}\label{R3223}
        R_{3223}=\ve_1 \ve_4 \lambda_4, \quad R_{4224}=\ve_1 \ve_3 \lambda_3, \quad R_{4334}=\ve_1 \ve_2 \lambda_2.
    \end{equation}
  For a $1$-stein $R$ we also have additional equations $\sum_i \ve_i R_{ixyi}=0$ for $1 \leq x \neq y \leq 4$ (see \cite{VKr}). Using these equations for $(x, y) \in \{(2, 3), (2, 4), (3, 4)\}$, the equations \eqref{R1}, \eqref{R2}, and \eqref{R4}, we conclude $R_{2443}=-\ve_1 \ve_4 R_{2113}$, $R_{2334}=-\ve_1 \ve_3 R_{2114}$, and $R_{3224}=-\ve_1 \ve_2 R_{3114}$. Thus, using \eqref{R2113}, we obtain
  \begin{equation}\label{R2443}
     R_{2443}=R_{2334}=R_{3224}=0. 
  \end{equation}
  Since Osserman $R$ is $2$-stein, the equation \eqref{trag} holds for $j=2$, so we get $\tr(\mc{J}_{E_1})^2 = (\ve_{E_1} )^2 c_2$, which gives 
  \begin{equation}\label{c2}
      \lambda_2^2+\lambda_3^2+\lambda_4^2=c_2.
  \end{equation}
  Since $R$ is $2$-stein, for all $1 \leq x \neq y \leq 4$ we get known $2$-stein equations (see \cite{VKr})
  \begin{equation*}
   2\sum\limits_{1\leq i, j \leq 4} \ve_i \ve_j R_{ixxj}R_{iyyj}+\sum\limits_{1\leq i, j \leq 4} \ve_i \ve_j(R_{ixyj}+R_{iyxj})^2=2\ve_x \ve_y c_2.   
  \end{equation*}
   For $(x, y)=(2, 3)$, using \eqref{R1}, \eqref{R2}, \eqref{R4}, \eqref{R2113}, \eqref{R1223}, and \eqref{R2443}, we get
   \begin{equation*}
   \begin{aligned}    2&\ve_1^2R_{2112}R_{3113}+2\ve_4^2R_{4224}R_{4334}+\ve_1\ve_4(R_{1234}+R_{1324})^2\\
       +&\ve_2\ve_3(-R_{3223})^2+\ve_3\ve_2(-R_{3223})^2+\ve_4\ve_1(R_{4231}+R_{4321})^2=2\ve_2\ve_3c_2.
   \end{aligned}
   \end{equation*}
   Using \eqref{R3223}, we compute $4\ve_2 \ve_3 \lambda_2\lambda_3+2\ve_2 \ve_3(R_{1234}+R_{1324})^2+2\ve_2 \ve_3 \lambda_4^2=2\ve_2\ve_3c_2$. Since $2 \ve_2 \ve_3 \neq 0$, we get $c_2-\lambda_4^2-2\lambda_2\lambda_3=(R_{1234}+R_{1324})^2$ and using \eqref{c2} we get
  \begin{equation*}
      (\lambda_3-\lambda_2)^2=(R_{1234}+R_{1324})^2.
  \end{equation*}
  Similarly, using the equations \eqref{R1}, \eqref{R2}, \eqref{R4}, and \eqref{R3} we obtain
  \begin{equation*}
  \begin{aligned}
      &(\lambda_2-\lambda_4)^2=(R_{1243}+R_{1423})^2=(R_{1324}-2R_{1234})^2,\\
      &(\lambda_4-\lambda_3)^2=(R_{1432}+R_{1342})^2=(R_{1234}-2R_{1324})^2.
      \end{aligned}
  \end{equation*}
  Hence, we get 
  \begin{equation}\label{s2s3s4}
  \begin{aligned}
       &s_4(\lambda_3-\lambda_2)=R_{1234}+R_{1324}, \, s_3(\lambda_2-\lambda_4)=R_{1324}-2R_{1234},\\
       &s_2(\lambda_4-\lambda_3)=R_{1234}-2R_{1324},
  \end{aligned} 
  \end{equation}
  where $s_2, s_3, s_4 \in \{-1, 1\}$. 
  According to the pigeonhole principle, at least two of $s_2, s_3, s_4$ are the same. First, suppose $s_i=s_j=-s_k$, where $(i, j, k)$ is a permutation of $(2, 3, 4)$. Summing the equations in \eqref{s2s3s4} we obtain
  \begin{equation*}      
  (s_3-s_4)\lambda_2+(s_4-s_2)\lambda_3+(s_2-s_3)\lambda_4=0,
  \end{equation*}
   and we conclude $(s_j-s_k)\lambda_i+(s_k-s_i)\lambda_j=0$, so $\lambda_i=\lambda_j$.
  Notice that substituting $s_k$ with $-s_k$ does not change the equations \eqref{s2s3s4} and provides $s_2=s_3=s_4$.

  If $s_2=s_3=s_4=-1$, then substituting eigenvectors $E_2$, $E_3$ and $E_4$ with $-E_2$, $-E_3$ and $-E_4$, respectively, we conclude that $R_{1234}$ and $R_{1324}$ change the sign, as well as $s_2$, $s_3$, $s_4$. Therefore, without loss of generality we can suppose $s_2=s_3=s_4=1$, and get the equations
  \begin{equation}\label{umestosistema}
  \begin{aligned}
      &R_{1234}-2R_{1324}=\lambda_4-\lambda_3, \\
      &R_{1324}-2R_{1234}=\lambda_2-\lambda_4,\\
      &R_{1234}+R_{1324}=\lambda_3-\lambda_2.
  \end{aligned}   
  \end{equation}
 
  For an arbitrary $Y \perp X=E_1$ there exist real numbers $k_2, k_3, k_4$ such that $Y=k_2E_2+k_3E_3+k_4E_4$, and therefore
  \begin{equation*} 
      \mc{J}_XY=\mc{J}_{E_1}(k_2E_2+k_3E_3+k_4E_4)
      =k_2 \ve_1 \lambda_2E_2+k_3 \ve_1 \lambda_3E_3+k_4\ve_1\lambda_4E_4.
  \end{equation*}
  Using the equations \eqref{R1}, \eqref{R2}, \eqref{R3}, \eqref{R4}, \eqref{R2113}, \eqref{R1223}, \eqref{umestosistema}, and $\mc{R}(X, Y)Z=\sum_i \ve_iR(X, Y, Z, E_i)E_i$ we calculate 
\begin{equation*}
  \begin{aligned}
      \mc{J}_YX&=\mc{J}_{k_2E_2+k_3E_3+k_4E_4}E_1=\mc{R}(E_1, k_2E_2+k_3E_3+k_4E_4)(k_2E_2+k_3E_3+k_4E_4)\\
      &=k_2 ^2 \ve_1 R_{2112}E_1+k_2k_3\ve_4 R_{1234}E_4+k_2k_4\ve_3R_{1243}E_3\\
      &+k_3k_2\ve_4 R_{1324}E_4+k_3^2\ve_1 R_{1331}E_1+k_3k_4 \ve_2R_{1342}E_2\\
      &+k_4k_2\ve_3 R_{1423}E_3+k_4k_3\ve_2R_{1432}E_2+k_4^2 \ve_1 R_{1441}E_1\\
      &=(k_2 ^2 \ve_1 R_{2112}+k_3^2\ve_1 R_{1331}+k_4^2 \ve_1 R_{1441})E_1+k_3k_4\ve_2(R_{1342}+R_{1432})E_2\\
      &+k_2k_4\ve_3(R_{1243}+R_{1423})E_3+k_2k_3\ve_4 (R_{1234}+R_{1324})E_4\\
      &=(k_2 ^2 \ve_2 \lambda_2+k_3^2\ve_3 \lambda_3+k_4^2 \ve_4 \lambda_4)E_1+k_3k_4\ve_2(R_{1234}-2R_{1324})E_2\\
      &+k_2k_4\ve_3(R_{1324}-2R_{1234})E_3+k_2k_3\ve_4(R_{1234}+R_{1324})E_4\\
      &=(k_2 ^2 \ve_2 \lambda_2+k_3^2\ve_3 \lambda_3+k_4^2 \ve_4 \lambda_4)E_1+k_3k_4\ve_2(\lambda_4-\lambda_3)E_2\\
      &+k_2k_4\ve_3(\lambda_2-\lambda_4)E_3
      +k_2k_3\ve_4(\lambda_3-\lambda_2)E_4.
  \end{aligned}
  \end{equation*}
  Thus, using that $(E_1, E_2, E_3, E_4)$ is an orthonormal basis, we compute
  \begin{equation*}
  \begin{aligned}
       &g(\mc{J}_XY, \mc{J}_YX)=k_2k_3k_4\ve_1 \ve_2 \lambda_2(\lambda_4-\lambda_3)g(E_2, E_2)\\
      &+k_2k_3k_4\ve_1 \ve_3 \lambda_3(\lambda_2-\lambda_4)g(E_3, E_3)+k_2k_3k_4\ve_1 \ve_4 \lambda_4(\lambda_3-\lambda_2)g(E_4, E_4)\\
      &=\ve_1 k_2 k_3 k_4 (\lambda_2(\lambda_4-\lambda_3)+\lambda_3(\lambda_2-\lambda_4)+\lambda_4(\lambda_3-\lambda_2))=0,
  \end{aligned}
  \end{equation*}
  which proves that $R$ is Jacobi-orthogonal.\\

  Conversely, let $R$ be a Jacobi-diagonalizable Jacobi-orthogonal algebraic curvature tensor of dimension $4$.
  First, we prove that $R$ is weak Jacobi-dual. Let $X$ and $Y$ be mutually orthogonal unit vectors in $\mc{V}$ such that $\mc{J}_XY=\ve_X \lambda Y$. Our aim is to prove $\mc{J}_YX=\ve_Y \lambda X$.  Since $R$ is Jacobi-diagonalizable and Jacobi-orthogonal, $X$ is nonnull and $Y \in \mc{V}(X)=\Ker(\widetilde{\mc{J}}_X-\ve_X \lambda \id)$, using Lemma \ref{BitnaLema1} and Lemma \ref{BitnaLema2}, we get $\mc{J}_YX=\ve_Y \lambda X + Z$, where $\ve_Z=0$ and $Z \in \mc{V}(X) \leq X^{\perp}$. Moreover, since $g(Z, Y)=g(\mc{J}_YX-\ve_Y \lambda X, Y)=g(X, \mc{J}_YY)-\ve_Y \lambda g(X, Y)=0$, it follows that $Z \perp Y$, so we conclude $Z \in \Span\{X, Y\}^{\perp}$.
  
  We discuss two cases. The case where $\Span\{X, Y\}^\perp$ is a definite subspace of $\mc{V}$ is easy since $\ve_Z=0$ and $Z \in \Span\{X, Y\}^{\perp}$ imply $Z=0$. 
  
  It remains to deal with the case where $\Span\{X, Y\}^{\perp}$ is indefinite ($\ve_X=\ve_Y$ for a Lorentzian $R$, $\ve_X=-\ve_Y$ for a Kleinian $R$, while for a Riemannian $R$ there is no such case). Since $\mc{J}_YX=\ve_Y \lambda X + Z$, our aim is to prove $Z=0$. We assume $Z \neq 0$, where $\ve_Z=0$ implies $Z$ is null. Since $R$ is Jacobi-diagonalizable, we know $\mc{V}(X)$ is nondegenerate such as $\Span\{Y\}^{\perp} \cap \mc{V}(X)$ which contains null vector $Z$, so its dimension is at least $2$. Thus, since $Y \in \mc{V}(X) \leq X^{\perp}$, we get $\dim \mc{V}(X)=3$. Therefore, $\mc{V}(X)=X^{\perp}$ and $\widetilde{\mc{J}}_X=\ve_X \lambda \id$. There exists $W \in \Span\{X, Y\}^{\perp}$ such that $\ve_W=-\ve_Y$ and we write $Y=(Y-tW)+tW$ for $t>1$. Since $Y-tW$, $tW \in \mc{V}(X)$, we have $\mc{J}_X(Y-tW)=\ve_X \lambda (Y-tW)$ and $\mc{J}_X(tW)=\ve_X \lambda tW$. Using $W \perp Y$ and \eqref{eps}, we get $\ve_{Y-tW}=\ve_Y+t^2\ve_W=(1-t^2)\ve_Y$ and $\ve_{tW}=t^2\ve_{W}$. Therefore $\sign(\ve_{Y-tW})=\sign(\ve_{tW})=-\sign(\ve_Y)$ and we apply the solved case to $X$, $Y-tW$ and $X$, $tW$ to obtain $\mc{J}_{Y-tW}X=\ve_{Y-tW}\lambda X$ and $\mc{J}_{tW}X=\ve_{tW}\lambda X$. Using the equation \eqref{jakobi} and $\mc{J}(tW, tW)X=\mc{J}_{tW}X$, we compute
  \begin{equation*}
      \begin{aligned}
          \mc{J}_YX&=\mc{J}_{(Y-tW)+tW}X=\mc{J}_{Y-tW}X+2\mc{J}(Y-tW, tW)X+\mc{J}_{tW}X\\
          &=\ve_{Y-tW}\lambda X+2t\mc{J}(Y, W)X-2\mc{J}_{tW}X+\mc{J}_{tW}X\\
          &=\ve_{Y-tW}\lambda X+2t\mc{J}(Y, W)X-\ve_{tW}\lambda X=\ve_{Y}\lambda X+2t\mc{J}(Y, W)X.
      \end{aligned}
  \end{equation*}
  Since $\mc{J}_YX=\ve_Y \lambda X+2t\mc{J}(Y, W)X$ holds for all $t>1$, we get $2\mc{J}(Y, W)X=0$ and $\mc{J}_YX=\ve_Y \lambda X$, contrary to assumption that $Z \neq 0$, so $Z=0$.

  Therefore, $R$ is weak Jacobi-dual and since $R$ is Jacobi-diagonalizable, using Lemma \ref{ARteo}, we conclude that $R$ is Jacobi-dual. Finally, Theorem \ref{NRteo} implies that $R$ is Osserman.
    \end{proof}

    Especially, since Riemannian curvature tensors are Jacobi-diagonalizable, we get that every algebraic curvature tensor on a positive definite scalar product space of dimension $4$ is Osserman if and only if it is Jacobi-orthogonal.

    At the end, we conclude that the Jacobi-orthogonal property is very important and useful in characterizing Osserman tensors in pseudo-Riemannian settings. 

\bibliographystyle{amsplain}

\end{document}